\documentclass[11pt, a4paper]{article}

\usepackage[utf8]{inputenc}
\usepackage[T1]{fontenc}
\usepackage{lmodern}
\usepackage{geometry}
\usepackage{amsmath, amssymb, amsthm, mathtools}
\usepackage{tikz-cd}
\usepackage{enumitem}
\usepackage[colorlinks=true, linkcolor=blue, citecolor=red, urlcolor=blue]{hyperref}
\usepackage[style=numeric, backend=biber, giveninits=true]{biblatex}
 \usepackage{xurl}
 \usepackage{hyperref}

\theoremstyle{plain}
\newtheorem{theorem}{Theorem}[section]
\newtheorem{lemma}[theorem]{Lemma}
\newtheorem{proposition}[theorem]{Proposition}

\theoremstyle{definition}
\newtheorem{definition}[theorem]{Definition}
\newtheorem{remark}[theorem]{Remark}

\newcommand{\Mon}{\mathbf{Mon}}
\newcommand{\Hopf}{\mathbf{Hopf}}
\newcommand{\op}{\mathrm{op}}
\newcommand{\Set}{\mathbf{Set}}
\newcommand{\Grp}{\mathrm{Grp}}
\newcommand{\com}{\mathrm{com}}

\newcommand{\Mod}{\mathbf{Mod}}

\title{On the Semi-Abelianness of Affine Group Schemes}
\author{David Forsman \\ 
\small Université catholique de Louvain \\
\small \texttt{david.forsman@uclouvain.be}}
\date{24 February, 2026}

\begin{document}

\maketitle

\begin{abstract}
    We prove that the category of commutative Hopf algebras over a field $k$ is co-semi-abelian. Consequently, the category of affine group $k$-schemes is semi-abelian. We establish coregularity by identifying the orthogonal factorization system of surjections and faithfully flat injections, and we deduce coexactness from Takeuchi’s correspondence between normal Hopf ideals and Hopf subalgebras of commutative Hopf $k$-algebras.
\end{abstract}

\section{Introduction}

It is a well-established result that the category of cocommutative Hopf $k$-algebras over an arbitrary field $k$ forms a semi-abelian category \cite{GRAN20194171, VespaWambst2018}, first proven for zero characteristic fields \cite{Gran_2015}. These results have motivated a broader investigation into the categorical properties of different Hopf structures; for instance, the semi-abelian condition has also been successfully analyzed for certain categories of cocommutative color Hopf algebras \cite{SCIANDRA2024107677}. The proof that cocommutative Hopf $k$-algebras form a semi-abelian category heavily relies on Newman's bijective correspondence between the Hopf subalgebras and the bi-ideals of cocommutative Hopf $k$-algebras \cite{Newman1975}. Our approach for the dual problem relies on Takeuchi's results about injective morphisms of commutative Hopf $k$-algebras being faithfully flat and Takeuchi's correspondence, which characterizes normal Hopf-ideals syntactically among Hopf-ideals \cite{Takeuchi1972}.

Introduced by Janelidze, Márki, and Tholen, semi-abelian categories provide a unified framework for homological algebra in non-abelian settings, possessing enough structure to support key lemmas and constructions in homological algebra \cite{JANELIDZE2002367}. The established semi-abelian nature of cocommutative Hopf $k$-algebras naturally raises the dual question: what is the categorical nature of commutative Hopf algebras? Coprotomodularity of commutative Hopf algebras was noted in the talk \cite{VanderLinden2018SplitExtBialgTalk} related to the paper by García-Martínez and Van der Linden \cite{GarciaMartinezVanDerLinden2018}. Homological self-duality of commutative Hopf algebras was shown in \cite{PeschkeVanDerLinden2024}.

We answer this question by proving that the category $\Hopf_k^\com$ of commutative Hopf algebras over a field $k$ is \textbf{co-semi-abelian}, meaning its opposite category of affine group schemes over $k$ (see \cite{Milne2012AffineGroupSchemes}) is a semi-abelian category. Furthermore, since Takeuchi's results extend to the context of super Hopf algebras \cite{Masuoka2005Fundamental}, the same proof demonstrates that the category of commutative super Hopf algebras over a field $k$ ($\mathrm{char}(k)\neq 2$) also forms a locally presentable co-semi-abelian category.

The paper is structured as follows. Section 2 covers the necessary categorical and algebraic preliminaries. In Section 3, we prove the main theorem for commutative Hopf algebras over a field.

\section{Preliminaries}

We remind the reader of the definition of semi-abelian categories and note how exactness can be seen as a stability property of normal subobjects (kernels) in a homological setting.

\begin{definition}
    We say that a category $C$ is:
    \begin{enumerate}
        \item \textbf{Pointed} if it has a zero-object, an object that is both initial and terminal.
        
        \item \textbf{Regular} if it has finite limits and every morphism $f$ factors $f = me$, where $m$ is a monomorphism and $e$ is a stably extremal epimorphism.\footnote{A family $(f_i\colon c_i\to d)_{i\in I}$ of morphisms is jointly extremal epimorphic, if a factorization $f_i = me_i$, where $m$ is a fixed monomorphism for all $i\in I$, guarantees that $m$ is an isomorphism. A morphism is an extremal epimorphism if its singleton family is jointly extremal epimorphic. A stably extremal epimorphism is a morphism whose all pullbacks are extremal epimorphisms.}\footnote{Extremal epimorphisms in regular categories are regular by \cite[Proposition 2.2 (b)]{Kelly1991}.}
        
        \item \textbf{Exact} (in the sense of Barr) if it is regular and every internal equivalence relation is a kernel pair of some morphism.
        
        \item \textbf{Protomodular} if pullbacks along retractions exist and for every retraction $r\colon  x \to y$ in $C$ with a section $s$ and a pullback $k\colon a \to x$ along $r$, the pair $(s, k)$ is jointly extremal epimorphic.
        $$
                \begin{tikzcd}
a \arrow[d, "k"'] \arrow[r] \arrow[rd, "\lrcorner", phantom, very near start] & b \arrow[d] \\
x \arrow[r, "r"']                                                             & y          
\end{tikzcd}
        $$
        
        \item \textbf{Semi-abelian} if it is pointed, protomodular, exact, and has binary coproducts. The category $C$ is \textbf{co-semi-abelian} if its dual $C^\op$ is semi-abelian.

        \item \textbf{Locally presentable}, if $C$ is equivalent to a reflexive subcategory of a pre-sheaf category $[K,\Set]$ closed under large enough directed colimits for some small category $K$ \cite{LPC}.
    \end{enumerate}
\end{definition}
As abelian categories are the same as exact additive ones \cite{BarrGrilletOsdol1971}, semi-abelian categories can be viewed as abelian categories where additivity is weakened to pointed protomodularity with binary coproducts. Protomodularity can be reflected between categories via a jointly conservative family $(F_i\colon C\to D_i)_i$ of pullback preserving functors. This allows an easy way to verify that the category $\Grp(C)$ of internal groups of a finitely complete category $C$ is protomodular via the functors $\Grp(C)\to \Grp(\Set)$ induced by the representables.

\begin{lemma}\label{lem:charexactness}
    Let $C$ be a pointed, regular, and protomodular category. Then the following are equivalent:
    \begin{enumerate}
        \item The category $C$ is exact.
        \item Regular epis map kernels to kernels: For every commutative square
        $$
    \begin{tikzcd}
a \arrow[d] \arrow[r, "f'", two heads] \arrow[d, "k"', hook] & b \arrow[d, "m", hook] \\
x \arrow[r, "f"', two heads]                                & y                     
\end{tikzcd}
    $$
    where $f$ and $f'$ are regular epis, $k$ a kernel and $m$ a monomorphism, $m$ is a kernel.
    \item For every reflexive relation $\underset{r_2}{\overset{r_1}{R\rightrightarrows x}}$ in $C$, the morphism $r_1k_2$ is a kernel, where $k_2$ is the kernel of $r_2$.
    \end{enumerate}
\end{lemma}
\begin{proof}
    Proof is essentially given in \parencite[381--382]{JANELIDZE2002367}.
\end{proof}

\section{Properties of Commutative Hopf Algebras}
Consider a symmetric monoidal category $A$ and consider the category $\Mon_A^\com$ of commutative monoids in $A$ with respect to the tensor product. Notice how the tensor product defines the coproduct in $\Mon_A^\com$. The category $\Hopf_A^\com$ of commutative Hopf $A$-algebras is defined as the category $\Grp((\Mon_A^\com)^\op)^\op$ of internal cogroup objects in $\Mon_A^\com$. As a shorthand, we write $\Hopf^\com_R$ for the category $\Hopf^\com_{\Mod_R}$ of Hopf $R$-algebras for a commutative ring $R$. It is known that $(\Hopf_R^\com)^\op$ is equivalent to the category of affine group schemes over a commutative ring $R$ \cite{Milne2012AffineGroupSchemes}. 
\begin{proposition}\label{prop:easyproperties}
    Let $A$ be a locally presentable symmetric monoidal category. Assume that the tensor product preserves directed colimits componentwise. Then the category of $\Hopf_A^\com$ of commutative Hopf $A$-algebras is a locally presentable, pointed, and protomodular category. 
\end{proposition}
\begin{proof}
    Local presentability of $\Mon_A^\com$ is shown in \cite{Porst01062008}. As coalgebraic structures within a locally presentable category form a locally presentable category by \cite[Remark 2.63]{LPC}, $\Hopf_A^\com$ is locally presentable. Since the representable functors induce a jointly conservative family of continuous functors $\Grp(C)\to \Grp(\Set)$ for any finitely complete category $C$, it follows that $\Hopf^\com_A$ is pointed coprotomodular.
\end{proof}

\subsection{Coregularity}

Coregularity is proven in two steps. First, we lift the orthogonal factorization system of surjections and injections from $\Mon^{\com}_k$ to $\Hopf^{\com}_k$, using the fact that injective morphisms in $\Mon^{\com}_k$ are closed under tensoring. The lifting of such factorization systems along monadic functors respecting the left class is established in \cite[Proposition 3.7]{Wissmann2022Minimality}, where it is shown that the properness assumption from Linton's original result \cite{Linton1969} is not necessary. With a slight modification, this proof extends to algebras in any finitely complete category, provided the left class is closed under products. Second,  we use the result by Takeuchi, which states that injective morphisms in $\Hopf^{\com}_k$ are faithfully flat, ensuring that injections are pushout-stable extremal monomorphisms.

\begin{definition}
    Let $f\colon R\to S$ be a morphism of commutative rings. We say that $f$ is faithfully flat if the induced functor $S\otimes_R(-)\colon \Mod_R\to \Mod_S$ between the module categories preserves finite limits and is faithful.
\end{definition}

\begin{theorem}[Faithful Flatness]
    Let $k$ be a field. Then any injective morphism $H_1\to H_2$ of commutative Hopf $k$-algebras is faithfully flat as a morphism of commutative rings. 
\end{theorem}
\begin{proof}
    A proof can be found in \cite[Theorem 3.1]{Takeuchi1972}.
\end{proof}

\begin{theorem}[Coregularity]\label{thm:coregularity}
    The category $\Hopf_k^\com$ is coregular, with surjective and faithfully flat morphisms being the classes of epimorphisms and extremal monomorphisms, respectively.
\end{theorem}
\begin{proof}
    Consider the orthogonal factorization system $(E,M)$ of $\Mon_k^\com$ of surjective and injective morphisms. As $k$ is a field, the class $M$ is closed under tensoring and it follows that $(E',M')$ is an orthogonal factorization system on $\Hopf^\com_k$, where $E' = U^{-1}(E)$, $M' = U^{-1}(M)$ and $U\colon \Hopf_k^\com\to \Mon^\com_k$ is the forgetful functor \cite[Proposition 3.7]{Wissmann2022Minimality}. By Takeuchi's theorem, every morphism in $M'$ is faithfully flat. Next, we demonstrate that the class $M'$ is pushout-stable. Consider a morphism $f\in M'$, its pushout $\bar{f}$ and the cokernel pair of $\bar{f}$ in $\Hopf_k^\com$. These pushouts are computed in the underlying category of $\Mon_k^\com$:
    $$
    \begin{tikzcd}
H \arrow[d, "f"'] \arrow[r, "g"] \arrow[rd, very near end, phantom, "\ulcorner"] & X \arrow[d, "\bar{f}" description] \arrow[r, "\bar{f}"] \arrow[rd, very near end, phantom, "\ulcorner"] & G\otimes_H X \arrow[d, "s_2"] \\
G \arrow[r, "\bar{g}"']                     & G\otimes_H X \arrow[r, "s_1"']                                                & G\otimes_H G\otimes_H X                
\end{tikzcd}
    $$
    Note that $G\otimes_H\bar{f} = s_1$ is a section and by faithfulness of $G\otimes_H(-)$ we witness the injectivity of $\bar{f}\colon X\to G\otimes_HX$. Thus $M'$ is pushout stable. For any epimorphism $f\in M'$, the morphism $G\otimes_Hf$ is an isomorphism as one of the morphisms of the cokernel pair of $f$. The faithfulness of $G\otimes_H(-)$ implies that $f$ is both a monic and epic morphism of modules and thus is an isomorphism. Hence, every morphism in $M'$ is a pushout stable extremal monomorphism, proving coregularity. Conversely, as every surjection is an epimorphism, we have by orthogonality that every extremal monomorphism is in $M'$. This shows that $E'$ and $M'$ are the classes of epimorphisms and extremal monomorphisms, respectively. 
\end{proof}

\subsection{Coexactness}
We prove the coexactness of $\Hopf_k^{\com}$. Again, by Takeuchi's work, we have a crucial characterization of normal Hopf ideals, and this is the key to showing the rest of the coexactness of $\Hopf_k^\com$. These observations assemble the necessary components to prove the co-semi-abelianness of $\Hopf_k^\com$.

\begin{definition}[Hopf ideal]
    Let $X$ be a commutative Hopf $k$-algebra and let $I$ be a subset of $X$. We call $I$ a Hopf ideal of $X$, if there is a morphism $f\colon X\to Y$ of $\Hopf_k^\com$, where $I = f^{-1}(0_Y)$. If $f$ may be chosen as a cokernel in $\Hopf^\com_k$, then $I$ is called a normal Hopf ideal of $X$.
\end{definition}
This notion of a Hopf ideal is equivalent to the standard one \cite{Montgomery1993}.

\begin{theorem}\label{thm:normalIdeals}
    Let $X$ be a commutative Hopf $k$-algebra with a Hopf ideal $I$ of $X$. Then the following are equivalent.
    \begin{enumerate}
        \item $I$ is a normal Hopf ideal.
        \item $I = XA^+$ for some Hopf subalgebra $A$ of $X$, where $A^+\coloneqq \varepsilon_A^{-1}(0)$ is the augmentation ideal of $A$ and $\varepsilon_A\colon A\to k$ is the augmentation of $A$.
        \item For each $x\in I$, we have $x_1S(x_3)\otimes x_2\in X\otimes I$.\footnote{The symbol $S$ denotes the antipode. We employ Sweedler's notation, where for the comultiplication we write $\Delta(x)$ as $x_1\otimes x_2$ and $(\Delta\otimes id)\circ \Delta = (id\otimes\Delta)\circ \Delta(x)$ as $x_1\otimes x_2\otimes x_3$.}
    \end{enumerate}
\end{theorem}
    
\begin{proof}
    The equivalence of the first two conditions is a straightforward verification, and the third condition is equivalent to the second by Theorem 4.3 in \cite{Takeuchi1972}.
\end{proof}

\begin{theorem}\label{thm:exactness}
    The category $\Hopf_k^\com$ is coexact.
\end{theorem}
\begin{proof}
Theorem \ref{thm:coregularity} shows that $\Hopf_k^\com$ is coregular. By Lemma \ref{lem:charexactness}, it suffices to show that the image along a regular epimorphism in $(\Hopf_k^\com)^{\op}$ of a normal subobject is normal. Consider the commutative diagram
    $$
    \begin{tikzcd}
X \arrow[r, "f'", two heads] \arrow[d, "k"', hook] & Z \arrow[d, "k'", hook] \\
Y \arrow[r, "f"', two heads]                       & W                    
\end{tikzcd}
    $$
    in $\Hopf_k$, where $f$ is a cokernel, $f'$ is an epimorphism, and $k, k'$ are inclusions. We show that $f'$ is a cokernel. Since $f'$ is surjective, it suffices to show that the Hopf ideal $I = f'^{-1}(0)$ is normal. Let $x\in I$. By Theorem \ref{thm:normalIdeals} it suffices to show that $x'\coloneqq x_1S(x_3)\otimes x_2\in X\otimes I$. Since $x\in J \coloneqq f^{-1}(0)$, we have that 
    $$
    x'\in (Y\otimes J)\cap (X\otimes X) = (Y\cap X)\otimes (J\cap X) = X\otimes I
    $$
    by the fact that intersections and tensor products satisfy interchange by the componentwise left-exactness of the tensor $\otimes_k$. Thus $I$ is a normal Hopf ideal, which shows that $f'$ is a cokernel.
\end{proof}

\subsection{Conclusion}

\begin{theorem}\label{thm:cosemiabelian}
    The category $\Hopf^\com_k$ of commutative Hopf algebras over a field $k$ is a co-semi-abelian category.
\end{theorem}
\begin{proof}
    The category is pointed, locally presentable, coprotomodular (Proposition \ref{prop:easyproperties}), and coexact (Theorem \ref{thm:exactness}).
\end{proof}

\begin{remark}
    The coexactness of $\Hopf_A^\com$, where $A$ is a symmetric monoidal abelian category with a componentwise exact tensor product, is proven the same way as Theorem \ref{thm:exactness}, assuming similar characterizations of faithfully flat morphisms and Hopf-ideals. Consider the symmetric monoidal category $A$ of super $k$-vector spaces over a field $k$ of characteristic other than $2$. As these characterizations of faithfully flat morphisms and normal Hopf-ideals have been worked out in \cite[Corollary 5.5, Theorem 5.9] {Masuoka2005Fundamental} for the category $\Hopf_A^\com$, we attain that the category $\Hopf^\com_A$ of commutative super Hopf $k$-algebras is co-semi-abelian.
\end{remark}
In a future paper, we aim to extend the semi-abelianness result to affine group schemes over commutative von Neumann regular rings by applying a local-to-global principle.
\section*{Acknowledgements}
This research was funded by the Fonds de la Recherche Scientifique (Belgium) through an Aspirant fellowship. The author thanks Professor Van der Linden for suggesting this problem.

\printbibliography

\end{document}